\documentclass[12pt,a4paper,twoside]{article}
\usepackage{graphicx}

\usepackage[english]{babel}
\usepackage[utf8]{inputenc}
\usepackage{amssymb,bbm,amsmath,amsthm,amscd,dsfont,empheq,cancel,mathrsfs}
\usepackage[margin=2cm]{geometry}
\usepackage{verbatim}
\usepackage{xcolor}

\newtheorem{descrip}{Description}
\newtheorem{teorema}{Theorem}
\newtheorem{definicion}{Definition}

\newtheorem{lema}{Lemma}
\newtheorem{remark}{Remark}
\def\disp{\displaystyle}
\def\adin{\alpha_{2in}}
\def\auout{\alpha_{1out}}
\def\atin{\alpha_{3in}}
\def\adout{\alpha_{2out}}
\def\fidin{\varphi_{2in}}

\def\fidout{\varphi_{2out}}

\title{A boundary-oriented reduced Schwarz domain decomposition technique for parametric advection-diffusion problems}

\begin{document}

\author{
 Manuel Bernardino del Pino \thanks{Departamento EDAN \& IMUS, Universidad de Sevilla, Spain  {\tt manubernipino@hotmail.com}}, \,
Tom\'as Chac\'on Rebollo \thanks{Departamento EDAN \& IMUS, Universidad de Sevilla, Spain  {\tt chacon@us.es}}, \,
Macarena G\'omez M\'armol \thanks{Departamento EDAN, Universidad de Sevilla, Spain  {\tt macarena@us.es}}}

\maketitle

\begin{abstract}
We present in this paper the results of a research motivated by the need of a very fast solution of thermal flow in solar receivers. These receivers are composed by a large number of parallel pipes with the same geometry.  We have introduced a reduced Schwarz algorithm that skips the computation in a large part of the pipes. The computation of the temperature in the skep domain is replaced by a reduced mapping that provides the transmission conditions. This reduced mapping is computed in an off-line stage.  We have performed an error analysis of the reduced Schwarz algorithm, proving that the error is bounded in terms of the linearly decreasing error of the standard Schwarz algorithm, plus the error stemming from the reduction of the trace mapping. The last error is asymptotically dominant in the Schwarz iterative process. We obtain $L^2$ errors below $2\%$ with relatively small overlapping lengths.
\end{abstract}

\section{Introduction}

This work is motivated by the need of building fast computational solvers for thermal flows within concentrated solar power (CSP) devices. Solar heat receivers for CSP nowadays are standard in sustainable energy production, by heating either air or molten salts with solar radiation collected by parabolic mirrors. Providing fast accurate numerical solvers of the thermal flow within these receivers is of paramount importance to optimise their design in order to ma\-xi\-mi\-se their thermal energy production. The solar heat receivers are composed by a large number of parallel pipes (several tens) with the same geometry, and connected flow at inflow (distributor) and outflow (collector) boundaries.  \\

We here focus the reduction of the computational cost involved, by a combined Reduced Order Modelling (ROM) - Domain Decomposition Method (DDM) strategy.  We use a reduced version of the alternating DDM, introduced by Schwarz in 1869  (cf. \cite{Schwarz}) for a Poisson problem on a domain composed of two simple ones, namely a disc and a rectangle. Its convergence was proved by P. L. Lions in \cite{Lions}. A thorough analysis and extensions to several DDMs with increased complexity can be found in the book of Tarek \cite{tarek}.\\

The DD-ROMs that have been built up to date, to the best of the authors' knowledge, focus on the computation of the unknown (in our case the temperature) in the full targeted domain. Let us mention, without intending to be exhaustive, the works \cite{baiges, baiyengar, li,nguyen, reyes, pasini} that deal with DD-ROMs for several kinds of PDE-governed models.  However, for solar receivers only the heat transported by the outflow is really needed to compute the overall thermal energy produced. \\

We introduce in this paper a reduced Schwarz algorithm that skips the computation in a large part of the domain, requiring only the solution of the advection-diffusion equation in two small domains, close to the inflow and outflow boundaries. The computation of the temperature in the skep domain is replaced by a reduced mapping yielding the transmission conditions between the resolved subdomains. This reduced mapping is computed in an off-line stage. This stage uses the information of a number of full-order Schwarz algorithm runs to build reduced spaces to approximate the traces via POD analysis. An artificial neuronal network is used to compute the input $\mapsto$ output mapping of the latent variables.\\

We have performed an error analysis of the reduced Schwarz algorithm for quite general linear advection-reaction-diffusion equations, proving that the error is bounded in terms of the linearly decreasing error of the standard Schwarz algorithm, plus the error stemming from the reduction of the trace mapping, which is asymptotically dominant in the Schwarz iterative process.\\

We have obtained a reduction of the errors, without needing to increase the overlapping, by two procedures. On one hand, by enriching the training data on the discrete trace mapping. On another hand,  by POD analysis with respect to stronger norms than the $L^2$ norm on the active boundaries in the Schwarz iterative process, actually with respect to discrete $H^1$ norms.\\

The paper is organised as follows. We describe the structure of the reduced Schwarz algorithm for advection-diffusion equations in Section \ref{algoritmo}. In Section \ref{se:redtran} we build the reduced mapping that yields the transmission conditions, for multi-dimensional advection-diffusion equations. Section \ref{se:error} carries on the error analysis, while Section \ref{se:numex} presents the results of several numerical tests. Finally Section \ref{se:concl} addresses some conclusions and future extensions of this work.
\section{A reduced Schwarz algoritm} \label{algoritmo}
\def\Gin{\Gamma_{in}}
\def\Gout{\Gamma_{out}}
\def\Gad{\Gamma_{ad}}
\def\Guup{\Gamma_{1up}}
\def\Gudown{\Gamma_{1down}}
\def\Guout{{\Gamma_{1out}}}
\def\Gtup{\Gamma_{3up}}
\def\Gtin{{\Gamma_{3in}}}
\def\Gtdown{\Gamma_{3down}}
\def\Gdin{{\Gamma_{2in}}}
\def\Gdout{{\Gamma_{2out}}}
\def\Gdup{\Gamma_{2up}}
\def\Gddown{\Gamma_{2down}}
\def\Guside{\Gamma_{1side}}
\def\Gtside{\Gamma_{3side}}
\def\Gdside{\Gamma_{2side}}

We consider a parametric advection-diffusion model problem that aims to model the heat flux along a pipe $\Omega= \omega_1 \times (0,L)\subset \mathbb{R}^d$ (with $d=1,\, d=2$ or $d=3$), where $\omega_1 $ is a bounded open subset of  $\mathbb{R}^{d-1}$ ($\omega_1=\emptyset$ when $d=1$). The parameter is the P\'eclet number $Pe$, that measures the relative strengths of inertial and diffusion forces. The domain $\Omega$ is assumed to have typically with a very large aspect ratio $L/\mbox{Diameter}(\omega_1)$. We assume the boundary of $\Omega$ to be Lipschitz continuous, and split into inflow, outflow and adiabatic boundaries respectively named $\Gin$, $\Gout$ and $\Gad$:
\begin{equation} \label{problema}
\left\{\begin{array}{lll}
\mathcal{L}(u,Pe) =f \text{ in } \Omega,\\
 u=0 \text{ on } \Gad,\\
 u=g \text{ on } \Gin, \\
\partial_n u=0 \text{ on } \Gout,
\end{array}\right.
\end{equation}
where $\mathcal{L}$ is the elliptic operator defined by $\mathcal{L}(u,Pe):=-\varepsilon\Delta u+\beta\cdot \nabla u$, $\beta\in\mathds{R}^d$ is the velocity, that we assume constant, satisfying $\beta \cdot \mathbf{n}\ge 0$ on $\Gout$, where $\mathbf{n}$ is the outer normal to $\partial \Omega$, $\varepsilon>0$ is the eddy diffusivity of the medium, and $f$, $g$ are the source and boundary data. If $f\in H^1(\Omega)'$ and $g \in H^{1/2}(\Gin)$ this problem admits a unique weak solution $u \in H^1_D(\Omega)=\{ v \in H^1(\Omega)\,|\, v_{|_{\Gamma_{ad}}}=0\,\} $ by the Lax-Milgram's Lemma. \\

Our objective is to parametrically solve this problem by reduced order methods when the P\'eclet number varies in a bounded interval $\cal D$. Since for multi-dimensional problems the P\'eclet \lq\lq number" is a vector of $d$ components, we consider a fixed $\beta \in\mathds{R}^d$ and set
$$
Pe = \frac{|\beta|\,H}{\varepsilon} 
$$
where $|\beta|$ is the euclidean norm of $\beta$ and $H=\mbox{Diameter}(\omega_1)$. 

 To describe the alternating Schwarz method that we consider, we perform a geometrical de\-com\-po\-si\-tion of the domain $\Omega$ into three overlapping subdomains $\Omega_1,\, \Omega_2$ and $\Omega_3,$ in the form $\Omega_1= \omega_1 \times (0,L_2)$, $\Omega_2= \omega_1 \times (L_1,L_4)$, $\Omega_3= \omega_1 \times (L_3,L)$ with $0<L_1 <L_2 <L_3 < L_4<L$ (see Figure \ref{fig:midibujo}). In this way,
\begin{equation*}
    \Omega=\text{int}\,(\overline{\Omega_1\cup \Omega_2 \cup \Omega_3}),\, \Omega_1 \cap \Omega_2= \omega_1 \times (L_1,L_2),\, \Omega_2 \cap \Omega_3= \omega_1 \times (L_3,L_4),\,\mbox{and  }\, \Omega_1 \cap \Omega_3=\emptyset.
\end{equation*}
Our purpose is to parameterise by reduced order modelling the solution (more concretely, the boundary transmission conditions) in $\Omega_2$ in the off-line stage, and to only compute the solution in $\Omega_1$ and $\Omega_3$ in the online stage. To reduce the computing time, we assume that the length of $\Omega_2$ is much larger than that of $\Omega_1$ and $\Omega_3$ and that the overlappings are of small size. 

We define the internal boundaries by
\begin{equation*}
\Gdin=\overline{\partial{\Omega_2} \cap \Omega_1}, \quad \Guout=\overline{\partial{\Omega_1} \cap \Omega_2}, \quad 
\Gtin=\overline{\partial{\Omega_3} \cap \Omega_2}, \quad 
\Gdout=\overline{\partial{\Omega_2} \cap \Omega_3}, \quad 
\end{equation*}
and the side boundaries by
$$
\Gamma_{1side}=\partial \omega_1\times [0,L_2],\quad\Gamma_{2side}=\partial \omega_1\times [L_1,L_4],\quad\Gamma_{3side}=\partial \omega_1\times [L1,L_4].
$$ 
 The boundaries $\Gdin$,  $\Guout$, $\Gtin$ and $\Gdout$ are the interfaces of the Schwarz domain decomposition method that we next describe. We denote $u^k_i$ the $k$-th iterate determined by the method in the domain $\Omega_i$, $i=1,2,3$:\\

\begin{figure}[h!]
    \centering
    \includegraphics[width=14cm]{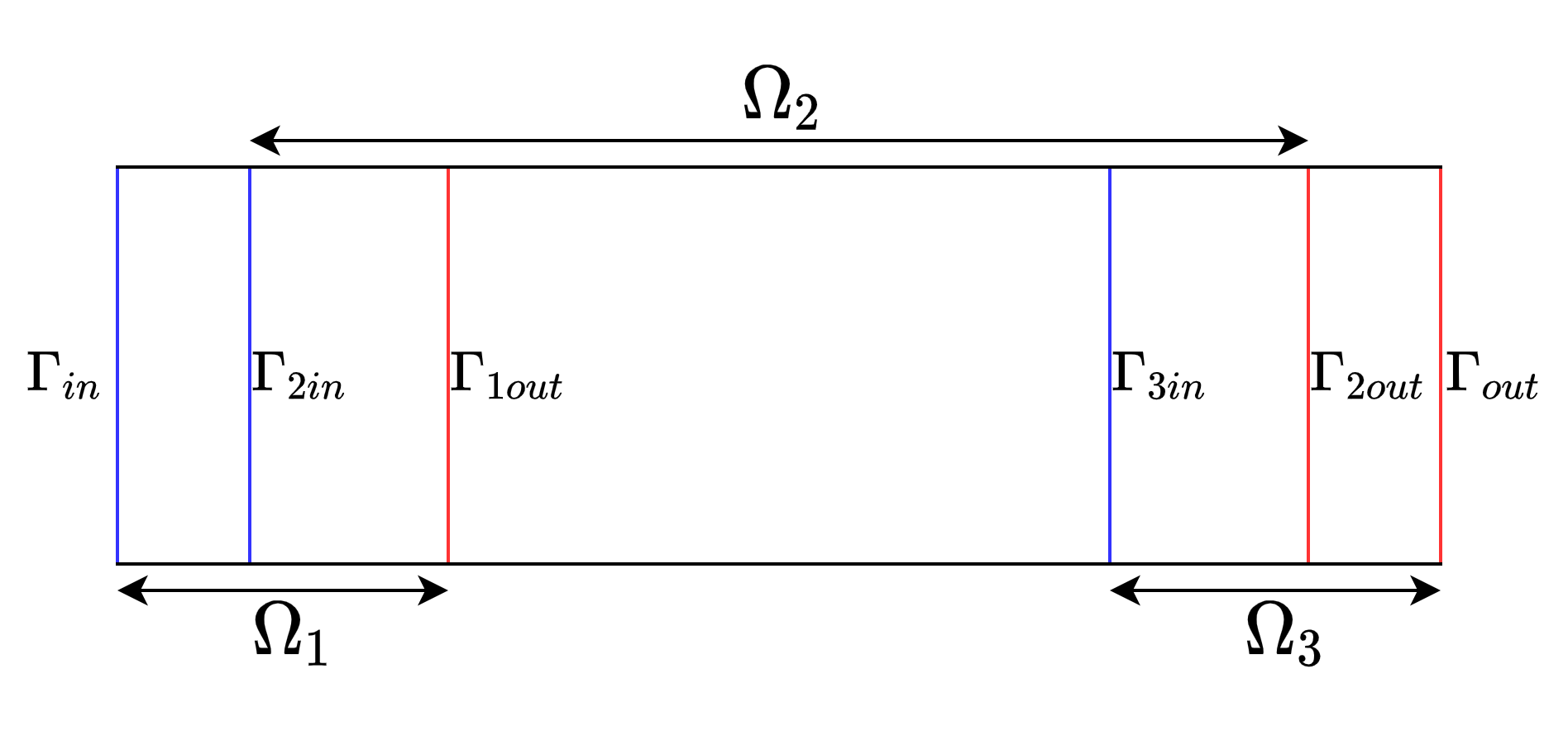}
    \caption{Decomposition of the domain $\Omega$ into three overlapping subdomains.}
    \label{fig:midibujo}
\end{figure}

Given a initial condition $u_2^0$ on $\Guout$ and $\Gtin$, we solve in $\Omega_1$, $\Omega_3$, 
\begin{equation}\label{ecomega13}
\left\{\begin{array}{lll}
\mathcal{L}(u_1^{k+1},Pe) =f \text{ in } \Omega_1,\\
 u_1^{k+1}=0 \text{ on } \Guside,\\
 u_1^{k+1}=g \text{ on } \Gin, \\
 u_1^{k+1}=u_2^k \text{ on } \Guout,
\end{array}\right.
\qquad 
\left\{\begin{array}{lll}
\mathcal{L}(u_3^{k+1},Pe) =f \text{ in } \Omega_3,\\
 u_3^{k+1}=0 \text{ on } \Gtside, \\
 u_3^{k+1}=u_2^k \text{ on } \Gtin,\\
\partial_n u_3^{k+1}=0  \text{ on } \Gout,
\end{array}\right.
\end{equation}
and we then solve in $\Omega_2$:
\begin{equation} \label{ecomega2}
\left\{\begin{array}{lll}
\mathcal{L}(u_2^{k+1},Pe) =f \text{ in } \Omega_2,\\
 u_2^{k+1}=0 \text{ on } \Gdside,\\
 u_2^{k+1}=u_1^{k+1} \text{ on } \Gdin, \\
 u_2^{k+1}=u_3^{k+1}\text{ on } \Gdout.
\end{array}\right.
\end{equation}
As we only are interested in knowing the amount of thermal energy at the outflow boundary, we do not need to compute the temperature $u$ in $\Omega_2$. If some reduced approximation $\tilde \tau$ of the parametric mapping 
$$
\tau:\,(u_{2_{|_\Gdin}}^{k},u_{2_{|_\Gdout}}^{k})(Pe) \in H^{1/2}(\Gdin)\times H^{1/2}(\Gdout)\mapsto (u_{2_{|_\Guout}}^{k},u_{2_{|_\Gtin}}^{k}) \in H^{1/2}(\Guout)\times H^{1/2}(\Gtin)$$ was known, only the computation of $u$ in the subdomains $\Omega_1$ and $\Omega_3$ would be needed to implement the Schwarz algorithm. Considering the transmission conditions on $\Gamma_{2in}$ and $\Gamma_{2out}$ in \eqref{ecomega2}, this reduced algorithm would read:

\vspace*{.3cm}\noindent
Given $u_1^{0}\in H^1(\Omega_1)$ and $u_3^0 \in H^1(\Omega_3)$ such that $u_1^{0}=0 \text{ on } \Guside$, $u_3^{0}=0 \text{ on } \Gtside$, iteratively solve 
\begin{equation} \label{algored}
\left\{\begin{array}{lll}
\mathcal{L}(u_1^{k+1},Pe) =f \text{ in } \Omega_1,\\
 u_1^{k+1}=0 \text{ on } \Guside,\\
 u_1^{k+1}=g \text{ on } \Gin, \\
 u_1^{k+1}= \tilde{\tau}_{1out}(u_1^k\big\vert_{\Gdin},u_3^k\big\vert_{\Gdout}) \text{ on } \Guout,
\end{array}\right.
\qquad 
\left\{\begin{array}{lll}
\mathcal{L}(u_3^{k+1},Pe) =f \text{ in } \Omega_3,\\
 u_3^{k+1}=0 \text{ on } \Gtside, \\
 u_3^{k+1}= \tilde{\tau}_{3in}(u_1^k\big\vert_{\Gdin},u_3^k\big\vert_{\Gdout})\,\,\text{ on } \Gtin, \\
\partial_n u_3^{k+1}=0  \text{ on } \Gout,
\end{array}\right.
\end{equation}
%
where $\tilde{\tau}_{1out}$ and $\tilde{\tau}_{3in}$ respectively denote the first and second components of $\tilde{\tau}$.
\hfill $\square$\\

This reduction procedure of the Schwarz algorithm, that we have described for the continuous problem \eqref{problema}, is readily extended to a numerical discretisation of this problem. It is enough to change the continuous by the discrete formulations of the boundary value problems \eqref{algored}. We do not describe it here in detail for brevity, although we use it in the numerical experiments section within a Galerkin finite element discretisation framework.

\subsection{One-dimensional heat flow}
To illustrate the basics of the reduced Schwarz method that we introduce, let us consider a \lq\lq toy" 1D advection-diffusion problem:

\begin{equation} \label{problema1d}
\left\{\begin{array}{lll}
\mathcal{L}(u,Pe) =f \text{ in } \Omega=(K,L),\\
 u(K)=\gamma, \quad u(L)=\mu,
\end{array}\right.
\end{equation}
where $\mathcal{L}(u,Pe):=-\varepsilon u''+\beta u',f\in\mathds{R}$ and $\beta \in \mathds{R}\setminus\{0\}$. The exact solution of this parametric problem is
\begin{equation}\label{solcd1d}
    u(x;K,L;\gamma,\mu)=\gamma+\frac{f}{\beta}(x-K)+\frac{\beta(\mu-\gamma)-f\,(L-K)}{\beta\left(e^{\frac{\beta}{\varepsilon}(L-K)}-1 \right)}\left(e^{\frac{\beta}{\varepsilon}(x-K)}-1 \right).
\end{equation}

In this case all trace spaces $H^{1/2}(\Gdin),\, H^{1/2}(\Gdout), \,H^{1/2}(\Guout)$ and $H^{1/2}(\Gtin)$ are equal to $\mathbb{R}$. Then $\tau$ is simply a mapping from $\mathbb{R}^2$ into itself,  analytically known from the above parametric expression of $u$, by 
\begin{equation}\label{solcd1d}
    \tau(\gamma,\mu)=(u(L_2;L_1,L_4;\gamma,\mu),\, u(L_3;L_1,L_4;\gamma,\mu)),\quad \forall \gamma, \mu \in \mathbb{R}.
\end{equation}
Then we may take the reduced trace mapping $\tilde \tau$ as the exact $\tau$.
In \cite{tesis} it is proved that for the Schwarz method for two subdomains, the error between consecutive iterates behaves as
$$
\sum_{i=1}^2\|u_i^{(k+1)}-u_i^{(k)}\|_{H^1(\Omega_i)}\simeq C\, e^{-2\delta Pe},
$$
where $\delta$ is the overlapping length. If we consider that $e^{(k)}\simeq C\,\rho^{-k}$, with
\begin{equation*}
 {e}^{(k)}:=\sum_{i=1}^2\|u_i^{(k+1)}-u_i^{(k)}\|_{H^1(\Omega_i)},\quad \rho=e^{2\delta Pe},
\end{equation*}
then $\rho^{(k)}:={e}^{(k)}/{e}^{(k+1)}$ is approximately equal to $e^{2\delta Pe}$, and then $\log(\rho^{(k)})/(2\,\delta)$ should approximately be equal to $Pe$. In Table \ref{tab1:1D} we check that the estimation given for two subdomains also is valid in the case of three subdomains. The convergence is extremely fast yet for moderate values of $Pe$ and $\delta$.

\begin{table}[ht]
\begin{center}
\begin{tabular}{| c | c | c | c|}
\hline
$Pe$ & $\delta$ & $\rho^{(k)}$ & $\log(\rho^{(k)})/2\delta$ \\ \hline
 & 1 & 6.8492 & 0.9620 \\
1 & 2 & 46.6656 & 0.9607 \\
 & 4 & 2.1464E+3 & 0.9589 \\
\hline
 & 1 & 46.6414 & 1.9212\\
2 & 2 & 2.1407E+3 & 1.9172 \\
 & 4 & 5.4006E+6 & 1.9377 \\
 \hline
  & 1 & 3.1626E+2 & 2.8782\\
3 & 2 & 9.7820E+4 & 2.8727 \\
 & 4 & 8.8299E+9 & 2.8626 \\
 \hline 
  & 1 & 2.1405E+3 & 3.8343\\
4 & 2 & 5.2806E+6 & 3.8698 \\
 & 4 & 2.6062E+13 & 3.8614 \\
 \hline
\end{tabular}
\caption{One-dimensional advection-diffusion operator. Numerical convergence rates of Schwarz algorithm with 3 overlapping subdomains.}
\label{tab1:1D}
\end{center}
\end{table}


As the trace mapping $\tau$ is exact, then the reduced Schwarz algorithm provides the same accuracy on $\Omega_1$ and $\Omega_3$, whitout needing to compute the solution on the large domain $\Omega_2$.

\section{Multi-dimensional advection-diffusion equations}\label{se:redtran}
In this section we study the construction of approximate reduced trace mappings $\tilde \tau$ for multi-dimensional advection-diffusion equations in an off-line stage. 

For multi-dimensional advection-diffusion equations, the sets 
$${\cal V}_{2in}=\{u_{2_{|_\Gdin}}^{k}(Pe)\, |\,Pe \in {\cal D},\,k=1,2,\cdots\}, \,\,{\cal V}_{2out}=\{u_{2_{|_\Gdout}}^{k}(Pe)\, |\,Pe \in {\cal D},\,k=1,2,\cdots\},
$$
$${\cal V}_{1out}=\{u_{2_{|_\Guout}}^{k}(Pe)\, |\,Pe \in {\cal D},\,k=1,2,\cdots\},\,\,{\cal V}_{3in}=\{u_{2_{|_\Gtin}}^{k}(Pe)\, |\,Pe \in {\cal D},\,k=1,2,\cdots\},
$$ respectively lie in varieties of $H^{1/2}(\Gdin)$, $H^{1/2}(\Guout)$, $H^{1/2}(\Gtin)$ and $H^{1/2}(\Gdout)$. As the Schwarz algorithm converges in $H^1(\Omega)$ norm, then the set ${\cal V}_{2in}$, for instance, is bounded in $H^{1/2}(\Gdin)$. Then, as the injection of $H^{1/2}(\Gdin)$ into $L^2(\Gdin)$ is compact, ${\cal V}_{2in}$ can be uniformly approximated by a finite-dimensional sub-space of $L^2(\Gdin)$ in $L^2(\Gdin)$ norm (cf. \cite{ocho}).

We intend to approximate these varieties by linear subspaces of these trace spaces, by means of a reduced order modelling procedure in an off-line stage. Then the mapping $\tau$ will be approximated by a transformation between the latent parameters of their representations in these linear subspaces. \\

The off-line stage consists of three steps, that we describe next:\\

{\bf Step 1: Construction of POD basis of the trace varieties.} 

{\bf a) }Consider a training finite set ${\cal D}_{train}\subset {\cal D}$ of given P\'eclet numbers. Given an initial condition $u^0\in H^1_D(\Omega)$, compute the full-order Schwarz method by solving \eqref{ecomega13}-\eqref{ecomega2} for each $Pe \in {\cal D}_{train}$ until some preset error threshold $\varepsilon$ between consecutive iterates is reached, that is, until
$$
\|u_i^{k^*+1}-u_i^{k^*}\|_{H^1(\Omega_i)} 
 < \varepsilon, \quad i=1,2,3
 $$
 for some $k^* \ge 1$. Construct the sets
 $${\cal V}_{2in}^*=\{u_{2_{|_\Gdin}}^{k}(Pe)\, |\,Pe \in {\cal D},\,k=1,2,\cdots,k^*\}, \,\,{\cal V}_{2out}^*=\{u_{2_{|_\Gdout}}^{k}(Pe)\, |\,Pe \in {\cal D},\,k=1,2,\cdots,k^*\},
$$
$${\cal V}_{1out}^*=\{u_{2_{|_\Guout}}^{k}(Pe)\, |\,Pe \in {\cal D},\,k=1,2,\cdots,k^*\},\,\,{\cal V}_{3in}^*=\{u_{2_{|_\Gtin}}^{k}(Pe)\, |\,Pe \in {\cal D},\,k=1,2,\cdots,k^*\},
$$
 {\bf b) }By Proper Orthogonal Decomposition (POD) of each of the sets ${\cal V}_{r}^*$, \break $r=2in, \, 2out,\, 1out, \, 3in$, with respect to some inner product $(\cdot,\cdot)_{\Gamma_r}$ in $H^{1/2}(\Gamma_r)$, build orthogonal subsets $\{\varphi^{(k)}_r, k=1,\cdots,\ell_r\,\}_{r}$ with respect to this inner product, in such a way that the rate of energy kept by the space spanned by these sets is above $1-\sigma$, for some small $\sigma >0$.\\
 
 {\bf Step 2. Construction of POD trace expansions coefficients.} 
 
  {\bf a) }For each $r =2in, \, 2out,\, 1out, \, 3in$, find the POD coefficients of all elements $u_{2_{|_{\Gamma_r}}}^{k}(Pe)$ of ${\cal V}_{r}^*$ with respect to the basis $\{\varphi^{(j)}_r, j=1,\cdots,\ell_r\,\}_{r}$, that is
 \begin{equation}\label{podcoefs1}
 u_{2_{|_{\Gamma_r}}}^{k}(Pe) \simeq \sum_{j=1}^{\ell_r} \alpha^{(j)}_r(k,Pe) \,\varphi^{(j)}_r.
 \end{equation}
 \\
  {\bf b) }To enrich the training set, locate the POD coefficients of the elements of  ${\cal V}_{2in}^*$ and ${\cal V}_{2out}^*$ between a range of minimum and maximum values, that is, $\alpha_r^{(j)}\in [\beta_r^{(j)},\mu_r^{(j)}]$ for $r=2in,\,2out$ and $j=1,\ldots,\ell_r$. Subsequently, consider a training subset $\mathcal{D}_r^{(j)}$ of each of the intervals $[\beta_r^{(j)},\mu_r^{(j)}]$, and solve in $\Omega_2$,  for each $\alpha_r^{(j)}\in \mathcal{D}_r^{(j)}$ and for each $Pe\in \mathcal{\tilde{D}}$.
as follows:
\begin{equation}\label{pbref2}
\left\{\begin{array}{lll}
\mathcal{L}(u_2,Pe) =f \text{ in } \Omega_2,\\
 u_2=0 \text{ on } \Gdside,\\
 u_2= \disp\sum_{j=1}^{\ell_{2in}} \adin^{(j)}\fidin^{(j)} \text{ on } \Gdin, \\
 u_2= \disp\sum_{j=1}^{\ell_{2out}} \adout^{(j)}\fidout^{(j)}\text{ on } \Gdout,
\end{array}\right.
\end{equation}
and represent the traces of $u_2$ in $\Guout$ and $\Gtin$, in their respective POD bases. Joining these coefficients with those already obtained by \eqref{podcoefs1}, obtain a set of coefficients associated to each POD basis index $j=1,\cdots,\ell_r$ of each boundary $\Gamma_r$,  $r=2in, \, 2out,\, 1out, \, 3in$: $\vec{v}^{(j)}
_r =\left \{\tilde{\alpha}^{(j,r)}_m \right\}_{m=1}^{L_r} \in \mathbb{R}^{L_{r}}$. \\
  \\
 {\bf Step 3. Construction of reduced trace mappings.} 
 
 By some supervised machine learning or artificial neural network approximation, compute respective extensions $\tilde{\tau}^{(j)}_{1out}$ and $\tilde{\tau}^{(j)}_{3in}$ of each of the mappings given by
 $$
 \left(\vec{v}^{(1)}_{2in},\vec{v}^{(2)}_{2in},\cdots ,\vec{v}^{(\ell_{2in})}_{2in}; \vec{v}^{(1)}_{2out},\vec{v}^{(2)}_{2out},\cdots ,\vec{v}^{(\ell_{2out})}_{2out}\right) \mapsto \vec{v}^{(j)}_{1out},\,\,\mbox{for}\,\, j=1,\cdots, \ell_{1out},
  $$
$$
\left(\vec{v}^{(1)}_{2in},\vec{v}^{(2)}_{2in},\cdots ,\vec{v}^{(\ell_{2in})}_{2in}; \vec{v}^{(1)}_{2out},\vec{v}^{(2)}_{2out},\cdots ,\vec{v}^{(\ell_{2out})}_{2out}\right) \mapsto \vec{v}^{(j)}_{3in},\,\,\mbox{for}\,\, j=1,\cdots, \ell_{3in}.
  $$
  
  \hfill$\square$

The mappings $\tilde{\tau}_1^{(j,1out)}$ and $\tilde{\tau}_2^{(j,3in)}$ transform the POD coefficients of the Dirichlet data in the sets ${\cal V}_{2in}$ and  ${\cal V}_{2out}$ into approximations of the POD coefficients of the traces of $u_2$ in $\Guout$ and $\Gtin$,
$$
 \tilde{\tau}^{(j)}_{1out}:\left(\adin^{(1)},\adin^{(2)},\cdots ,\adin^{(l_{2in})}; \adout^{(1)},\adout^{(2)},\cdots ,\adout^{(l_{2out})}\right) \mapsto \alpha^{(j)}_{1out},\,\,\mbox{for}\,\, j=1,\cdots, \ell_{1out};  $$
 $$
 \tilde{\tau}^{(j)}_{3in}:\left(\adin^{(1)},\adin^{(2)},\cdots ,\adin^{(l_{2in})}; \adout^{(1)},\adout^{(2)},\cdots ,\adout^{(l_{2out})}\right) \mapsto \alpha^{(j)}_{3in},\,\,\mbox{for}\,\, j=1,\cdots, \ell_{3in}.
  $$
This allows to compute an approximation, that we denote by $\tilde{\tau}$, of the mapping between traces $\tau$ as follows: Given $(u_{2_{|_\Gdin}},u_{2_{|_\Gdout}}) \in H^{1/2}(\Gdin)\times H^{1/2}(\Gdout)$, we set
$$
 \tilde{ \tau}\left (u_2\big\vert_{\Gdin},u_2\big\vert_{\Gdout}\right ) = \left (\sum_{j=1}^{\ell_{1out}}\alpha^{(j)}_{1out}\, \varphi^{(j)}_{1out},\, \sum_{j=1}^{\ell_{3in}}\alpha^{(j)}_{3in}\, \varphi^{(j)}_{3in}\right ) \in H^{1/2}(\Guout)\times H^{1/2}(\Gtin);
  $$
  where
 $$
\alpha^{(j)}_{r}= \tilde{\tau}^{(j)}_{r}\left(\adin^{(1)},\adin^{(2)},\cdots ,\adin^{(l_{2in})}; \adout^{(1)},\adout^{(2)},\cdots ,\adout^{(l_{2out})}\right)  ,\,\,\mbox{for}\,\, j=1,\cdots, \ell_{r},\,\, r=1out,\, 3in;  $$
  with
  $$
\alpha^{(j)}_r=(u_{2_{|_{\Gamma_r}}}, \varphi_{r}^{(j)})_{\Gamma_r},\, \,\,\mbox{for}\,\,j=1,\cdots, \ell_{2r},\,r=2in,\,2out.$$\hfill $\square$\\
 

Observe that this off-line procedure intends to reduce the full Schwarz iterative process, rather than reducing the solution in each sub-domain and then applying the standard Schwarz algorithm, just replacing the full-order discretisation spaces by the reduced ones. Actually, it is possible to build reduced spaces approximating the solutions in $\Omega_1$ and $\Omega_3$ and then to combine both reduction procedures.

\section{Error analysis} \label{se:error}
We next perform the error analysis of the reduced Schwarz method introduced in Section \ref{algoritmo}. We prove that if the error in the approximation of the mapping $\tau$ by the mapping $\tilde{\tau}$ is small enough, then the error due to the Schwarz method is bounded by the error made in this approximation, plus the error due to the unperturbed method.

 The error analysis actually applies to a general advection-diffusion-reaction problem,

\begin{equation}\label{eqprincipal}
\left\{\begin{array}{ccc}
\mathcal{L}u \equiv -\nabla \cdot (a\,\nabla u)+\nabla\cdot (\beta\, u)+c\,u =f \text{ in } \Omega\\
 u=0 \text{ on } \Gamma_D \\
 \textbf{n}\cdot (a\nabla u)+\gamma\,u=g_{N} \text{ on } \Gamma_N,
\end{array}\right.
\end{equation}
on a bounded domain $\Omega \subset \mathds{R}^d$ (where $d\ge 1$ is an integer number), with Lipschitz-continuous boundary $\partial \Omega$, with outer normal $\textbf{n}$; Dirichlet boundary  $\Gamma_D \subset \partial\Omega$ of positive $d-1$-dimensional measure, and natural (Neumann or Robin) boundary $\Gamma_N\subset \partial\Omega$ where $\overline{\Gamma_D} \cup \overline{\Gamma_N}=\partial\Omega$ and $\Gamma_D \cap \Gamma_N=\emptyset$. We shall assume that the diffusion coefficient $a \in L^\infty(\Omega)$, the driving velocity $\beta(x)\in L^\infty(\Omega)^d$, the reaction rate $c(x) \in L^\infty(\Omega)$ and the Robin condition coefficient $\gamma \in L^\infty(\Gamma_N)$ satisfy:
\begin{equation*}
    a_0|\xi|^2\leq \xi^{T}a(x)\xi \leq a_1|\xi|^2, \, \, \,\forall \,\xi\in \mathds{R}^d \mbox{  a. e.}\, x\in \Omega,\,\,    
\end{equation*}
for some $a_1\ge a_0 >0$;
\begin{equation*}
     \nabla \cdot \beta=0,\,\, c \ge 0 \, \mbox{  a. e. in}\, \Omega,\,\, \gamma \ge 0 \, \mbox{  a. e. in}\, \Gamma_N.
\end{equation*}
This ensures the coercivity of problem \eqref{eqprincipal}. We further assume data $f \in \left(H^{1}(\Omega)\right )'$, $g_N \in H^{-1/2}(\Gamma_N)$.
\\
We look for a weak solution of problem  in the subspace $V $ of $H^1(\Omega)$ defined by
\begin{equation*}
    V:=\{v\in H^1(\Omega) : v=0 \text{ on } \Gamma_D \}.
\end{equation*}
The variational formulation of \eqref{eqprincipal} seeks for a solution $u\in V$ that satisfies
\begin{equation} \label{weakform}
    \mathcal{A}(u,v)=F(v), \quad \forall v\in V,
\end{equation}
where $\mathcal{A}(\cdot,\cdot), \, F(\cdot)$ are defined by:
\begin{equation*}
    \mathcal{A}(u,v) := \int_{\Omega} (a\,\nabla u\cdot \nabla v+\nabla\cdot(\beta\, u) v +c\,u v)\, dx + \int_{\Gamma_N} \gamma\, u \,v \,ds(x), \text{ for } u,\,v\in V.
\end{equation*}
\begin{equation*}
    F(v) := \langle f,\, v \rangle_\Omega + \langle g_N,\, v \rangle_{\Gamma_N}, \text{ for } v\in V.
\end{equation*}
It is standard that by Lax-Milgram's Lemma, problem (\ref{weakform}) admits a unique solution.

Let us consider a non-overlapping $\Sigma_1,\ldots,\Sigma_p$ and an overlapping $\Omega_1,\ldots,\Omega_p$ decomposition of  $\Omega$ into open subsets, in the sense that $\overline\Omega=\displaystyle \overline{\Sigma_1\cup\ldots\cup\Sigma_p}$ and the intersection of two different $\Sigma_i$ is either the empty set or a set with zero measure in $\mathds{R}^d$, while $\overline\Omega=\displaystyle \overline{\Omega_1\cup\ldots\cup\Omega_p}$ and the intersection of two different $\Omega_i$ is either the empty set or a set with positive measure. We assume that these two decomposition are related by $\Sigma_i \subset \Omega_i$, $i=1,\cdots,p$. Let $\Gamma^{(i)}:=\partial \Omega_i\cap \Omega$ and $\Gamma_{[i]}:=\partial \Omega_i\cap \partial \Omega$ denote the interior and exterior boundary segments of $\Omega_i$. \\

Within the settings of Section \ref{algoritmo}, $p=3$, $\Sigma_1=\omega_1 \times (0,L_1)$, $\Sigma_2=\omega_1 \times (L_1,L_3)$, $\Sigma_3=\omega_1 \times (L_3,L)$, while $\Gamma^{(1)}=\Guout$, $\Gamma^{(2)}=\Gdin\cup \Gdout$, $\Gamma^{(3)}=\Gtin$. We adapt this notation to ease the error analysis.\\

Following \cite{tarek}, the sequential Schwarz alternating method for the solution of \eqref{eqprincipal} may be stated as follows. Each iteration (or sweep) will consist of $p$ fractional steps. We denote the iterate in the $i$'th fractional step of the $k$'th sweep as $w^{(k+\frac{i}{p})}$. Given $w^{(k+\frac{i-1}{p})}$ the next iterate $w^{(k+\frac{i}{p})}$ is the variational solution of

\begin{equation}\label{eqprincipal2}
\left\{\begin{array}{cccc}
 -\nabla \cdot \left(a\,\nabla w^{(k+\frac{i}{p})}\right)+\cdot(\beta\, \nabla w^{(k+\frac{i}{p})})+c\,w^{(k+\frac{i}{p})} =f \text{ in } \Omega_i\\
 \textbf{n}\cdot \left(a\nabla w^{(k+\frac{i}{p})}\right)+\gamma\,w^{(k+\frac{i}{p})}=g_{N} \text{ on } \Gamma_{[i]}\cap\Gamma_N, \\
 w^{(k+\frac{i}{p})}=w^{(k+\frac{i-1}{p})} \text{ on } \Gamma^{(i)},  \\
 w^{(k+\frac{i}{p})}=0 \text{ on } \Gamma_{[i]}\cap\Gamma_D.
\end{array}\right.
\end{equation}
The local solution $ w^{(k+\frac{i-1}{p})}$ is then extended outside $\Omega_i$ as follows:
\begin{equation*}
     w^{(k+\frac{i}{p})}\equiv  w^{(k+\frac{i-1}{p})}, \,\text{ on } \, \Omega\setminus \overline{\Omega}_i.
\end{equation*}

The following result states the convergence of this Schwarz algorithm (cf. \cite{tarek}).

\begin{teorema}\label{convergencia1}
Under the above conditions, the iterates $w^k$ will converge geometrically to the solution $u$ with:
\begin{equation*}
    \|u-w^{(k)}\|_{H^1(\Omega)} \leq \rho^k \|u-w^{(0)}\|_{H^1(\Omega)}.
\end{equation*}
for some $\rho\in (0,1)$.
\end{teorema}
The convergence factor $0<\rho<1$ generally depends on the overlap between the subdomains, the diameters \textrm{diam}($\Omega_i)$ of the subdomains and the data in \eqref{eqprincipal}.  

The updates $w^{(k+\frac{i}{p})}$ in the sequential Schwarz method can be expresed in terms of certain projection operators onto subspaces of $V$. To each $\Omega_i$ we associate a subspace $V_i$ of $V$ as:
\begin{equation*}
    V_i:=\{v\in V : v=0 \text{ in } \Omega \setminus \overline{\Omega_i}\}.
\end{equation*}
Let us consider the elliptic projection operator $P_i$ onto subspace $V_i$ of $V$ as follows.
\begin{definicion}
    Given $w\in V$ define $P_i w\in V_i$ as the solution of:
    \begin{equation*}
    \mathcal{A}(P_i w,v)=\mathcal{A}(w,v), \, \text{ for } v\in V_i.
    \end{equation*}
\end{definicion}

In \cite{tarek} it is proved that

\begin{lema} \label{propProj}
Let $u\in V$ be the solution of  \eqref{weakform}. Given $w \in V$, let $w_i \in V$ satisfying
\begin{equation}\label{partprob}
\left\{\begin{array}{cccc}
 -\nabla \cdot \left(a\,\nabla w_i\right)+\nabla\cdot(\beta\, \nabla w_i)+c\,w_i =f \text{ in } \Omega_i\\
 \textbf{n}\cdot \left(a\nabla w_i\right)+\gamma\, w_i=g_{N} \text{ on } \Gamma_{[i]}\cap\Gamma_N, \\
 w_i=w \text{ on } \Gamma^{(i)}  \\
 w_i=0 \text{ on } \Gamma_{[i]}\cap\Gamma_D,
\end{array}\right.
\end{equation}
 where $g_N=\textbf{n}\cdot (a\,\nabla u) + \gamma\, u \in H^{-1/2} (\Gamma_N)$, with $w_i\equiv w$ on $\Omega\setminus \Omega_i$.\\
Then $w_i=w+P_i(u-w).$
\end{lema}

The Schwarz alternating method may now reformulated in terms of the projection operators $P_i$. By Lemma \ref{propProj} with $w_i \equiv w^{(k+\frac{i}{p})}$ and $w \equiv w^{(k+\frac{i-1}{p})}$ it follows 
\begin{equation*}
 w^{(k+\frac{i}{p})}= w^{(k+\frac{i-1}{p})}+P_i\left(u- w^{(k+\frac{i-1}{p})}\right).
\end{equation*}
Substituting this representation into the Schwarz alternating method yields the projection formulation.\\

We next formulate an approximate Schwarz alternating method in which we consider that the transmission conditions in the internal boundaries $\Gamma^{(i)}$ experience an error generated by their approximation by the off-line step stated in Section \ref{algoritmo}. We thus change $w_i=w \text{ on } \Gamma^{(i)}$ by $w_i={\mathcal{R}_i}w \text{ on } \Gamma^{(i)}$, being ${\mathcal{R}_i}$ an operator that approximates the trace on $\Omega_i$, but considered as an operator from $V$ into $V$.  To determine ${\mathcal{R}_i}$, firstly we consider the operator $\tilde{\tau}$ that is actually built in the off-line stage described in Section \ref{algoritmo}, as a mapping from \break $ H^1( \Omega) \rightarrow H^{1/2}(\Guout)\times H^{1/2}(\Gtin) $ (possibly non-linear), given by
$$
\tilde{\tau}(u)=  (\psi_1,\psi_2)\in H^{1/2}(\Guout)\times H^{1/2}(\Gtin), \,\,\mbox{ with} \,\, \psi_i=\tilde{ \tau_i}\left (\gamma_{\Guout}(u),\gamma_{\Gtin}u\right ) ,\quad \forall u\in V,
$$
where $\gamma_{\Guout}$ and $\gamma_{\Gtin}$ respectively stand for the traces on $\Guout$ and $\Gtin$.  We then define the operators $\tilde{\mathcal{R}_j}(u):V\mapsto H^{1/2}(\partial\Omega_j)$, $j=1,3$ by
$$
\tilde{\mathcal{R}_1}(u)=\psi_1 \,\,\mbox{in }\, \Guout \, \,\tilde{\mathcal{R}_1}(u) = \gamma_1(u),\,\,\mbox{in }\, \partial\Omega_1\setminus \{\Guout\};$$
$$
 \tilde{\mathcal{R}_3}(u)=\psi_2 \,\,\mbox{in }\,\Gtin, \, \,\tilde{\mathcal{R}_3}(u) = \gamma_3(u)\,\,\mbox{in }\, \partial\Omega_3\setminus \{\Gtin\},
$$
where $\gamma_j$ denotes the trace operator on $\partial \Omega_j$. If needed, $\gamma_j(u)$  should  be replaced by a smoothed trace in an arbitrarily small neighbourhood of the corners of $\partial \Omega_j$ to ensure that ${\mathcal{R}_j}(u) \in H^{1/2}(\partial \Omega_j)$. Note that $(\gamma_1 \circ \tilde{\mathcal{R}_1)}(u)=\tilde\tau_{1out}(u)$, $(\gamma_3 \circ \tilde{\mathcal{R}_3)}(u)=\tilde\tau_{3in}(u)$.

 We then in general consider for $p$ subdomains that we have trace reduction operators $ \tilde{\mathcal{R}_i}:V\mapsto H^{1/2}(\partial\Omega_i)$ similar to similar to $\tilde{\mathcal{R}_j}$, $j=1,3$ that provide approximations of the traces in $\partial \Omega_i$. We shall assume that the relative error of the reduction is uniformly bounded, that is, there exists $\mu>0$ such that
 \begin{equation}\label{estir}
 \|(\gamma_i-\tilde{\mathcal{R}_i})(u)\|_{H^{1/2}(\partial \Omega_i)}\leq \mu \cdot \|u\|_V\quad \forall \, u\in V,
 \end{equation}
 where $\gamma_i :\,V \mapsto H^{1/2}(\partial\Omega_i)$ is the trace operator on $\partial\Omega_i$.
 
 We next consider a standard bounded linear lifting operator 
 ${\varepsilon_i}: H^{1/2}(\partial \Omega_i) \rightarrow H^1(\Omega)$ such that  $(\gamma_i\circ{\varepsilon_i}) \varphi=\varphi$ for all $\varphi \in H^{1/2}(\partial \Omega_i)$.  Let us define $\mathcal{R}_i:=I-{\varepsilon_i} \circ (\gamma_i-\tilde{\mathcal{R}_i}) : \,V \mapsto V$. Then, it holds 
$$
\gamma_i({\mathcal{R}_i} v )=\gamma_i(v)-(\gamma_i\circ\varepsilon_i)( \gamma_i v- \tilde{\mathcal{R}_i}(v))=\gamma_i(v)-( \gamma_i v- \tilde{\mathcal{R}_i}(v))= \tilde{\mathcal{R}_i}(v).
$$

So, if we change the condition $w_i=w $ by $w_i={\mathcal{R}_i}w \text{ on } \Gamma^{(i)}$ in problem \eqref{partprob} (observe that the notation in \eqref{partprob} skips the trace operator), with the present notation we are imposing $\gamma_i w_i=\tilde{\mathcal{R}_i}w \text{ on } \Gamma^{(i)}$. \\

Note that Lemma \ref{propProj} holds with $w_i={\mathcal{R}}_i w+P_i(u-{\mathcal{R}}_iw)$. This allows to prove a generalisation of Theorem \ref{convergencia1} including the error due to the approximation of the transmission conditions. To prove this theorem we need some previous results, stated next. We denote by ${\cal L}(E,F)$ denote the space of bounded linear operators from a normed space $E$ into another normed space $F$. When $E=F$ we denote ${\cal L}(E,F)$ by ${\cal L}(E)$. 
\begin{lema} Under condition \eqref{estir}, there exists a constant $C_1 >0$ such that 
\begin{equation} \label{estiri}
\|(I- {\mathcal{R}}_i)(v)\|_V \le C_1\, \mu \,\|v\|_V,\,\, \forall v \in V,\quad \forall i=1,\cdots, p.
\end{equation}
\end{lema}

\begin{proof}
As the operators $\varepsilon_i$ are bounded, combining with \eqref{estir}
it follows
 \begin{eqnarray*}
 \|(I- {\mathcal{R}}_i)(v)\|_{V}&\leq& \|{\varepsilon_i} \|_{{\cal L}(H^{1/2}(\partial \Omega_i) ,H^1(\Omega))}\, \|\gamma_i v- \tilde{\mathcal{R}_i}(v)\|_{H^{1/2}(\partial \Omega_i)} \leq C_1\, \mu \, \|v\|_{V},
 \end{eqnarray*}
  where $C_1=\displaystyle\max_{i=1,\cdots,p}\{\|\tilde{\varepsilon_i} \|_{
{\cal L}(H^{1/2}(\partial \Omega_i) ,H^1(\Omega))}\} $.
\end{proof}
Let us recall that from \cite{tarek}, $
\|I-P_i\|_{{\cal L}(V)} <1,\,\forall i =1,\cdots, p.
$

\begin{lema} Assume that 
\begin{equation}\label{propmu}
\displaystyle \mu <\frac{1-\rho}{C_1\,\rho}, \,\,\mbox{where  } \, \rho=\displaystyle \max_{i=1,\cdots,p} \|I-P_i\|_{{\cal L}(V)}.
\end{equation}
 Then, the sequence $\{w^{(k+\frac{i}{p})}\}_{k=0,1,\cdots; i=0,1,\cdots,p}$ is bounded in $V$.
\end{lema}
\begin{proof} 
According to Lemma \ref{propProj}, with $w_i = w^{(k+\frac{i}{p})}$, $w = {\mathcal{R}_i}w^{(k+\frac{i-1}{p})}$, we have:
\begin{equation*}
w^{(k+\frac{i}{p})}={\mathcal{R}_i}w^{(k+\frac{i-1}{p})}+P_i(u-{\mathcal{R}_i}w^{(k+\frac{i-1}{p})}), \quad \text{ for all } i=1,\ldots,p.
\end{equation*}
Using the linearity of the operator $P_i$ and subtracting and adding $w^{(k+\frac{i-1}{p})}$ and $u$, we can write:

\begin{equation} \label{errores}
\begin{split}
u-w^{(k+\frac{i}{p})} &=u-{\mathcal{R}_i}w^{(k+\frac{i-1}{p})}-P_i u+P_i{\mathcal{R}_i}w^{(k+\frac{i-1}{p})}=(I-P_i)(u-{\mathcal{R}_i}w^{(k+\frac{i-1}{p})}) \\
&= (I-P_i)(u-{\mathcal{R}_i}w^{(k+\frac{i-1}{p})}-w^{(k+\frac{i-1}{p})}+w^{(k+\frac{i-1}{p})})\\
&=(I-P_i)(u-w^{(k+\frac{i-1}{p})})+(I-P_i)(I-{\mathcal{R}_i})w^{(k+\frac{i-1}{p})},
\end{split}
\end{equation}
for each $i=1,\ldots,p$. 
Let $\tau=\rho\,(1+C_1\, \mu)$. Due to \eqref{propmu}, $\tau <1$. Denote $ e^{(k+\frac{i}{p})}= u-w^{(k+\frac{i}{p})}$. Then, using \eqref{estiri},
\begin{eqnarray*}
 \|e^{(k+\frac{i}{p})}\|_V &\le& \|I-P_i\|_{{\cal L}(V)}\,  (\|e^{(k+\frac{i-1}{p})}\|_V+ \|(I-{\mathcal{R}_i})w^{(k+\frac{i-1}{p})}\|_V )\nonumber \\
 &\le& \rho\, (\|e^{(k+\frac{i-1}{p})}\|_V+ C_1\,\mu\,\|w^{(k+\frac{i-1}{p})}\|_V )\le \tau \, \|e^{(k+\frac{i-1}{p})}\|_V+ C_1\,\rho\, \mu\,\|u\|_V )\nonumber \\
 &\le& ... \le \tau^{p(k+1)}\, \|e^{(0)}\|_V+ C_1\,\rho\, \mu\,(\tau^{p(k+1)-1}+\cdots+\tau^2+\tau+1)\, \|u\|_V )\\
 &\le&  R:=\|e^{(0)}\|_V+ C_1\,\rho\, \displaystyle \frac{1}{1-\tau}\, \|u\|_V ,
 \end{eqnarray*}
 where to obtain the third inequality we have added and substracted $u$ to $w^{(k+\frac{i-1}{p})}$, and we have used recursively the third estimate to obtain the fifth one. We thus conclude that the sequence $\{w^{(k+\frac{i}{p})}\}_{k=0,1,\cdots; i=0,1,\cdots,p}$ lies in a ball in $V$ of center $u$ and radious $R$.

\end{proof}

 The following generalisation of Theorem \ref{convergencia1} holds:\\

\begin{teorema}
 Assume that estimates \eqref{estir} and \eqref{propmu} hold. Then we have:
\begin{equation*}
    \|u-w^{(k)}\|_V \leq \rho^{k}\|u-w^{(0)}\|_V+\sigma \, \frac{\rho}{1-\rho}\,\mu,
\end{equation*}
for some $\sigma>0$ .
\end{teorema}

\begin{proof}

Writing the last identity in \eqref{errores} for $i=1,\ldots p$, we have:
\begin{equation*}
\begin{split}
    u-w^{(k+\frac{1}{p})}&=(I-P_1)(u-w^{(k)})+(I-P_1)(I-\mathcal{R}_1)w^{(k)}.\\
 u-w^{(k+\frac{2}{p})}&=(I-P_2)(u-w^{(k+\frac{1}{p})})+(I-P_2)(I-\mathcal{R}_2)w^{(k+\frac{1}{p})}.\\
\vdots  \\
 u-w^{(k+1)}&=(I-P_p)(u-w^{(k+\frac{p-1}{p})})+(I-P_p)(I-\mathcal{R}_p)w^{(k+\frac{p-1}{p})}.
\end{split}
\end{equation*}

Therefore
\begin{equation*}
 u-w^{(k+\frac{2}{p})}=(I-P_2)(I-P_1)(u-w^{(k)})+(I-P_2)(I-P_1)(I-\mathcal{R}_1)w^{(k)}+(I-P_2)(I-\mathcal{R}_2)w^{(k+\frac{1}{p})},
\end{equation*}
and applyig recursivity:
\begin{equation}\label{operadores}
\begin{split}
u-w^{(k+1)}&=(I-P_p)\cdots(I-P_2)(I-P_1)(u-w^{(k)})+(I-P_p)\cdots(I-P_2)(I-P_1)(I-\mathcal{R}_1)w^{(k)} \\
&+(I-P_p)\cdots(I-P_2)(I-\mathcal{R}_2)w^{(k+\frac{1}{p})}+\cdots+(I-P_p)(I-\mathcal{R}_p)w^{(k+\frac{p-1}{p})}.
\end{split}
\end{equation}
Consequently, it follows that

\begin{equation*}
u-w^{(k+1)}=T_1(u-w^{(k)})+\sum_{i=1}^p T_i  (I-{\mathcal{R}_i}) w^{(k+\frac{i-1}{p})},
\end{equation*}
where $T_i=(I-P_i)(I-P_{i+1})\cdots (I-P_p), \quad \forall i=1,\ldots, p$. 

Using \eqref{estir}, we have
\begin{equation*}
\|u-w^{(k+1)}\|_V \leq \|T_1\|_{\mathscr{L}(V)}\|u-w^{(k)}\|_V+ C_1\, \mu\, \sum_{i=1}^p \|T_i\|_{\mathscr{L}(V)}  \|w^{(k+\frac{i-1}{p})}\|_V
\end{equation*}
As $\|T_i\|_{\mathscr{L}(V)}\leq \rho$ for all 
$i=1,\ldots,p$, then
\begin{equation*}
\|u-w^{(k+1)}\|_V \leq \rho\, \|u-w^{(k)}\|_V+\rho \, C_1\, \mu \cdot\sum_{i=1}^p \|w^{(k+\frac{i-1}{p})}\|_V.
\end{equation*}
the sequence $\{w^{(k+\frac{i}{p})}\}_{k=0,1,\cdots; i=0,1,\cdots,p}$ is bounded in $V$, there exists $C>0$ such that 
\begin{equation*}
\|u-w^{(k+1)}\|_V \leq \rho \,\|u-w^{(k)}\|_V+\rho \,\sigma\,\mu,
\end{equation*}
where $\sigma:= C\,C_1\, p   $. Denoting $\alpha_{k+i}=\|u-w^{(k+i)}\|_V (i=0,1)$ we have $\alpha_{k+1}\leq \rho\alpha_k+\rho \tau$, with $\tau = \sigma\,\mu$. Recursively using this inequality it follows
\begin{equation*}
\begin{split}
\alpha_{k+1} &\leq \rho\alpha_k+\rho \tau \leq \rho (\rho \alpha_{k-1}+\rho\tau)+\rho \tau=\rho^2\alpha_{k-1}+(\rho^2+\rho)\tau=\rho^2(\rho \alpha_{k-2}+\rho \tau)+(\rho^2+\rho)\tau \\
&=\rho^3\alpha_{k-2}+(\rho^3+\rho^2+\rho)\tau 
\leq \cdots \leq \rho^{k+1}\alpha_0+(\rho^{k+1}+\rho^k+\cdots +\rho)\tau \\
&=\rho^{k+1}\alpha_0+\rho(1+\rho+\cdots+\rho^k)\tau=\rho^{k+1}\alpha_0+\rho\frac{1-\rho^{k+1}}{1-\rho} \tau.
\end{split}
\end{equation*}
Then, we arrive at the targeted estimate,
\begin{equation*}
    \|u-w^{(k+1)}\|_V \leq \rho^{k+1}\|u-w^{(0)}\|_V+\sigma\,\frac{\rho}{1-\rho}\, \mu.
\end{equation*}
\end{proof}
\begin{remark}
This result states that if there is a uniformly bounded perturbation of the transmission conditions, the error due to the Schwarz alternating method is essentially due to the perturbation, and with size of the same order, after a number large enough of iterations.
\end{remark}
\begin{remark}
Note that in the previous proof the uniform relative reduction error estimate \eqref{estir} for the trace reduction may be limited to a bounded set that contains the iterates \break $\{w^{(k+\frac{i}{p})}\}_{k=0,1,\cdots; i=0,1,\cdots,p}$.
\end{remark}

\section{Numerical Experiments}\label{se:numex}
In this section we test  the practical performances of the reduced Schwarz algorithm that we have introduced in the preceding section. We solve the advection-diffusion problem \eqref{problema} in the 2D pipe-like domain 
$$\Omega=(0,5) \times (0,40).$$
We assume there is no source term, that is, $f=0$, and that we have an input parabolic profile $g(x)=\frac{4}{25}x(5-x)$. We take  $\varepsilon=1$, and $\beta=(P,0.2)$, then the advection velocity has a non-vanishing vertical component. Then the P\'eclet number is 
$$
Pe= 5\,\sqrt{P^2+0.2^2}\simeq 5P.
$$

To apply the Schwarz method, we take $\Omega_1=(0,12)\times (0,5)$, $\Omega_2=(7,31)\times (0,5)$, $\Omega_3=(26,40)\times (0,5)$, so the overlapping length between $\Omega_1$ and $\Omega_2$ and also between $\Omega_2$ and $\Omega_1$ is $5$. We also take as initial condition $u_2=0.5*u$, where $u$ is the solution of problem \eqref{problema}, to keep the Schwarz iterates within some neighbourhood of the exact solution. In practice we solve transient problems, that lead to advection-reaction-diffusion problems after time iterations. To solve these problems the solution at the preceding step can be used as initial condition for the Schwarz iterates.\\

We have solved problem \eqref{problema} and the Schwarz iterates appearing in \eqref{ecomega13} and \eqref{ecomega2} with the finite element method, using piecewise linear elements and structured triangulations. The triangulations have been built in such a way that these coincide on the overlapping subdomains. The number of grid nodes is 15251 for $\Omega_1$, 10201 for $\Omega_2$ and 30401 for $\Omega_3$. We have used the FreeFem++ code to perform the computations. \\

We next describe the practical implementation of the off-line stage. We have considered two inner products to build the POD expansions on a given piece of boundary $\Gamma$ (either $\Gdin,\, \Gdout,\, \Guout$ or $\Gtin$): The $L^2(\Gamma)$ inner product, and a discrete $H^1(\Gamma)$ inner product, computed as follows. Assume that the grid nodes on $\Gamma$ are $\{(\overline x,y_k)\}_{k=0}^m$, for some $\overline x \in {\mathbb R}$, where we suppose $y_i=k \,h$, $k=0,\cdots,m$ with $h=5/m$. Then the discrete $H^1(\Gamma)$ inner product is defined as
$$
((v_h, w_h))=h\, \sum_{k=0}^m v_h(\overline x,y_k)\, w_h(\overline x,y_k) + h\, \sum_{k=1}^m \frac{v_h(\overline x,y_k)- v_h(\overline x,y_{k-1})}{h}  \,\frac{w_h(\overline x,y_k)- w_h(\overline x,y_{k-1})}{h} ,
$$
for any $v_h$, $w_h \in C^0(\overline{\Gamma})$. The natural choice for this inner product would have been the $H^{1/2}(\Gamma)$ one, but as this is quite complex to compute, we have preferred to use the above discrete $H^1(\Gamma)$ one for its simplicity.
\\

{\bf Step 1: Construction of POD basis of trace varieties:}
We take $P\in\mathcal{D}_{train}$, where $\mathcal{D}_{train}$ is a uniform partition with 50 nodes of the interval $\mathcal{D}=[5,14]$. We set $\ell$ as the number of basis functions that ensure that the rate of energy kept by the truncated POD expansion up to $l$ modes satisfies $\varepsilon(\ell)<1-\sigma$, where $\sigma=10^{-5}$. We obtain $\ell=2$ on all the boundaries $\Guout,\, \Gdin, \,\Gdout$ and $\Gtin$.\\

{\bf Step 2: Construction of POD trace expansions coefficients:} The coefficients $\{\alpha_r^{(j)}\}_{j=1}^\ell$ given by \eqref{podcoefs1} belong, with respect to the discrete $H^1(\Gamma)$ inner product, to the following intervals: $\alpha_{2in}^{(1)}\in[4.9876,7.0792], \alpha_{2in}^{(2)}\in[-0.1156,0.1346], \alpha_{2out}^{(1)}\in[0.7404,3.5738], \alpha_{2out}^{(2)}\in[-0.0144,0.0071]$. Similarly, with respect to the discrete $L^2(\Gamma)$ inner product: $\alpha_{2in}^{(1)}\in[4.2147,5.9895], \alpha_{2in}^{(2)}\in[-0.0707,0.0837], \alpha_{2out}^{(1)}\in[0.6248,3.0171], \alpha_{2out}^{(2)}\in[-0.0090,0.0045]$.
We have taken in each of the intervals $[\beta_{2in}^{(1)},\nu_{2in}^{(1)}]$, $[\beta_{2out}^{(1)},\nu_{2out}^{(1)}]$ into seven equally spaced values, $[\beta_{2in}^{(2)},\nu_{2in}^{(2)}]$, $[\beta_{2out}^{(2)},\nu_{2out}^{(2)}]$  into three equally spaced values and $\tilde{\mathcal{D}}$ is a uniform partition of 30 nodes of $\mathcal{D}$. Then, solve problem \eqref{pbref2} with the boundary conditions that appear therein.
\\

{\bf Step 3: Construction of approximate trace mappings:}
Finally, to approximate the mapping $\tau:(\adin^{(k)},\adout^{(k)}) \longrightarrow (\auout^{(k)} ,\atin^{(k)})$ for $k=1,2$, we have used an artificial neural network with a hidden layer of ten intermediate neurons and the activation function is the sigmoid function. We have used the Matlab R2022a code to perform the computations.\\

{\bf Online stage:} Using the above off-line procedure and data, we solve the reduced Schwarz method \eqref{algored} using as initial condition $u_1^{0}\big\vert_{\Guout}=u_3^{0}\big\vert_{ \Gtin}=0.5*u$. We stop the iterative process when the difference between two consecutive approximate solutions has a norm in $H^1$ less than $\varepsilon=10^{-9}$:
\begin{equation*}
    \|u_1^{(k+1)}-u_1^{(k)}\|_{H^1(\Omega_1)}+\|u_3^{(k+1)}-u_3^{(k)}\|_{H^1(\Omega_3)} < \varepsilon.
\end{equation*}

In Figure \ref{graficaerror1} we show the respective relative errors of the solution, $e_h=u_h - u$ measured in the norms
$$
\|e_h\|_{L^2(\Omega_i), rel}= \frac{ \|e_h\|_{L^2(\Omega_i)}}{\|u\|_{L^2(\Omega_i)}},\quad i=1,3,
$$
 computed for one hundred uniformly-spaced Péclet numbers in $[5,14]$. We show the errors corresponding to the POD expansions of Step 1 in the off-line stage computed with both $L^2(\Gamma)$ and the discrete $H^1(\Gamma)$ inner products. \\

It can be observed that considering the POD with respect to the discrete $H^1(\Gamma)$ inner product, these errors are smaller than those obtained with the POD w.r.t. the $L^2(\Gamma)$ inner product. Concretely, we approximately obtain an error reduction of $25\%$ in $\Omega_1$ and $10\%$ in $\Omega_3$.  The error reduction possibly occurs because the advection is a directional derivative, so the $H^1(\Gamma)$ norm will provide a better approximation.The larger error in $\Omega_3$ also is likely due to a decreasing in the precision of the reduced trace mapping produced by advective effects.\\

\begin{figure}[h!]
    \centering
    \includegraphics[width=8cm]{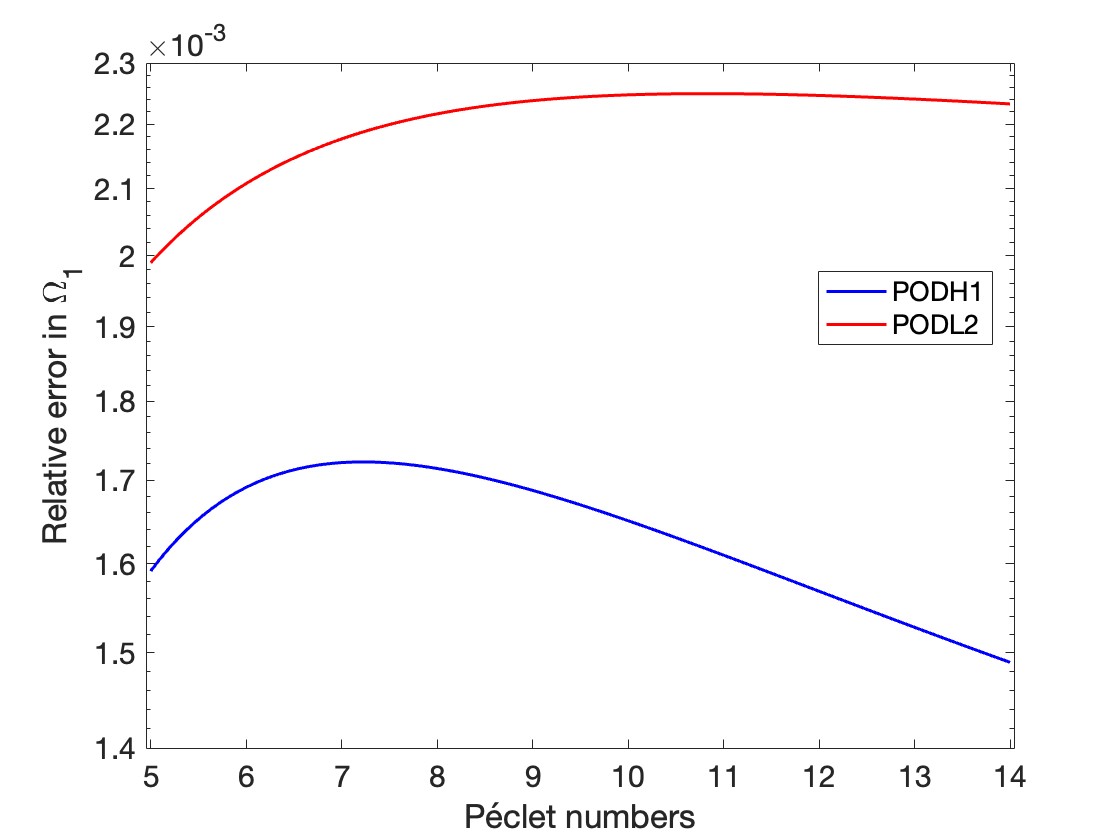}
        \includegraphics[width=8cm]{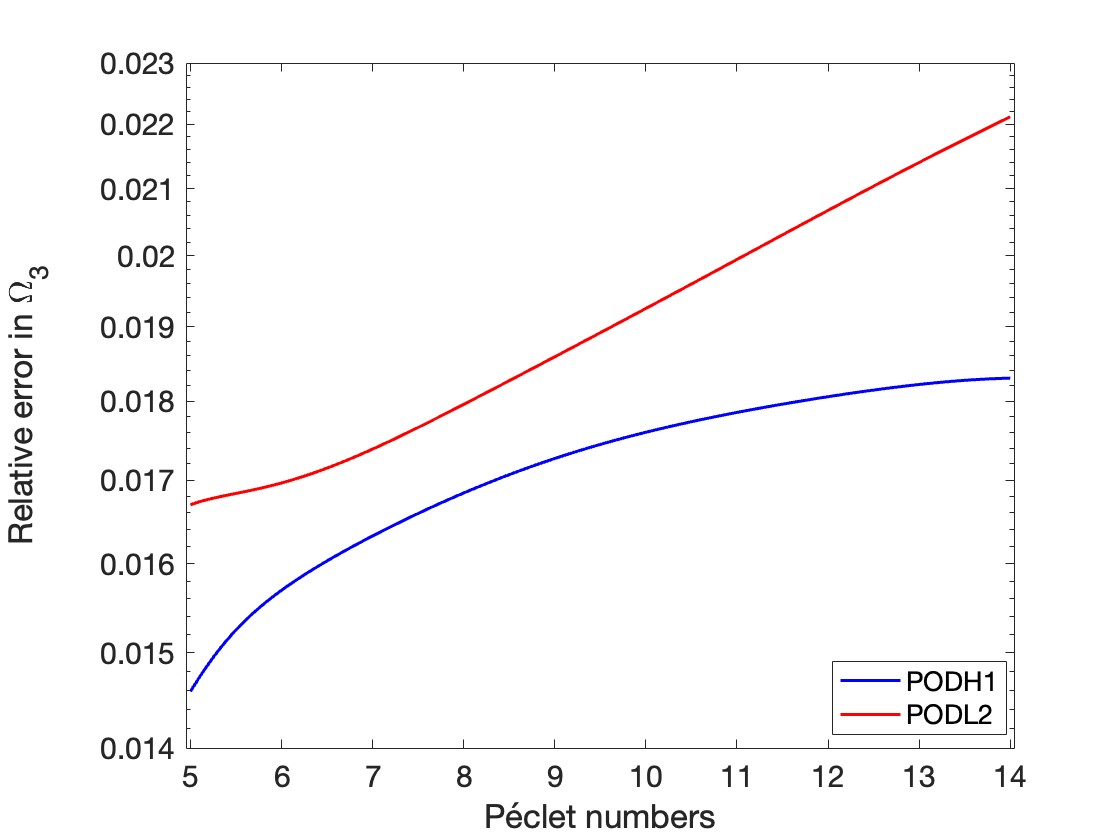}
    \caption{Reduced Schwarz algorithm. Relative errors in $L^2(\Omega_1)$ (left) and $L^2(\Omega_3)$ (right).}
    \label{graficaerror1}
\end{figure}


Table \ref{tab:errores} shows the relative errors (with POD in $H^1(\Gamma)$ inner product) for some Péclet numbers that do not belong to the training set. We observe that in all cases the relative errors are below $2\%$.\\

{\bf Checking the error reduction procedures.}Table \ref{tab:erroresin} displays the relative errors for the same P\'eclet numbers, but without the two procedures introduced to increase the accuracy of the reduced Schwarz mapping: here we use the $L^2(\Gamma)$, instead of the $H^1(\Gamma)$ inner product, to build the PODs; moreover, we skip the Step 2 of the off-line stage, that is, we do not enrich the sampling of the latent variables to build the discrete trace mapping $\tilde{\tau}$. We observe that the errors are much larger, particularly in subdomain $\Omega_3$, where the errors reach nearly $25\%$.\\

\begin{table}[ht]
\begin{center}
\begin{tabular}{| c | c | c |}
\hline
$Pe$ & $L^2(\Omega_1)$ & $L^2(\Omega_3)$ \\ \hline
5.09091       & 1.60456E-3 & 1.47253E-2 \\
6.18182  & 1.70066E-3  & 1.58263E-2 \\
7.27273 & 1.72251E-3 & 1.64748E-2 \\
8.36364 & 1.70619E-3  & 1.70081E-2\\
9.45455  & 1.67146E-3 & 1.74302E-2\\
10.5455  & 1.62873E-3  & 1.77462E-2\\
11.6364 & 1.58349E-3  & 1.79895E-2\\
12.7273 & 1.53883E-3  & 1.81814E-2\\
13.9091 & 1.49307E-3 & 1.83006E-2\\
\hline
\end{tabular}
\caption{Reduced Schwarz algorithm. Relative errors for some trial Péclet numbers.}
\label{tab:errores}
\end{center}
\end{table}

\begin{table}[ht]
\begin{center}
\begin{tabular}{| c | c | c |}
\hline
$Pe$ & $L^2(\Omega_1)$ & $L^2(\Omega_3)$ \\ \hline
5.09091       & 1.6004E-3 & 1.75438E-2 \\
6.18182  & 1.65095E-3  & 1.59138E-2 \\
7.27273 & 1.799696E-3 & 2.27712E-2 \\
8.36364 & 2.24054E-3  & 5.98679E-2\\
9.45455  & 2.29640E-3 & 9.95777E-2\\
10.5455  & 3.78726E-3  & 1.14010E-1\\
11.6364 & 4.58634E-3  & 1.17816E-1\\
12.7273 & 5.30954E-3  & 2.12517E-1\\
13.9091 & 5.99227E-3 & 2.45307E-1\\
\hline
\end{tabular}
\caption{Reduced Schwarz algorithm. Relative errors for trial Péclet numbers, without error decreasing procedures.}
\label{tab:erroresin}
\end{center}
\end{table}
{\bf Checking the error dependency of the overlapping length.} In order to check the error dependency on the the overlap size, we have performed the aforementioned calculations taking now $\Omega_1=(0,17)\times (0,5)$, $\Omega_2=(7,31)\times (0,5)$, $\Omega_3=(21,40)\times (0,5)$. Namely, the overlapping length between $\Omega_1$ and $\Omega_2$ and also between $\Omega_2$ and $\Omega_1$ is $10$. In Figure \ref{graficaerror1_solap10} we observe that the errors are somewhat smaller than in the case of the overlapping with length 10 (approximately 50\% in $\Omega_1$ and $75\% in \Omega_3$). We observe that in this case using the $H^1(\Gamma)$ inner product instead of the $L^2(\Gamma)$ one to build the POD analysis still improves the errors in $\Omega_3$, possibly again due to a better treatment of the convective effects. Thus, there is no need of a large overlapping to obtain acceptable errors from a practical point of view.

\begin{figure}[h!]
    \centering
    \includegraphics[width=8cm]{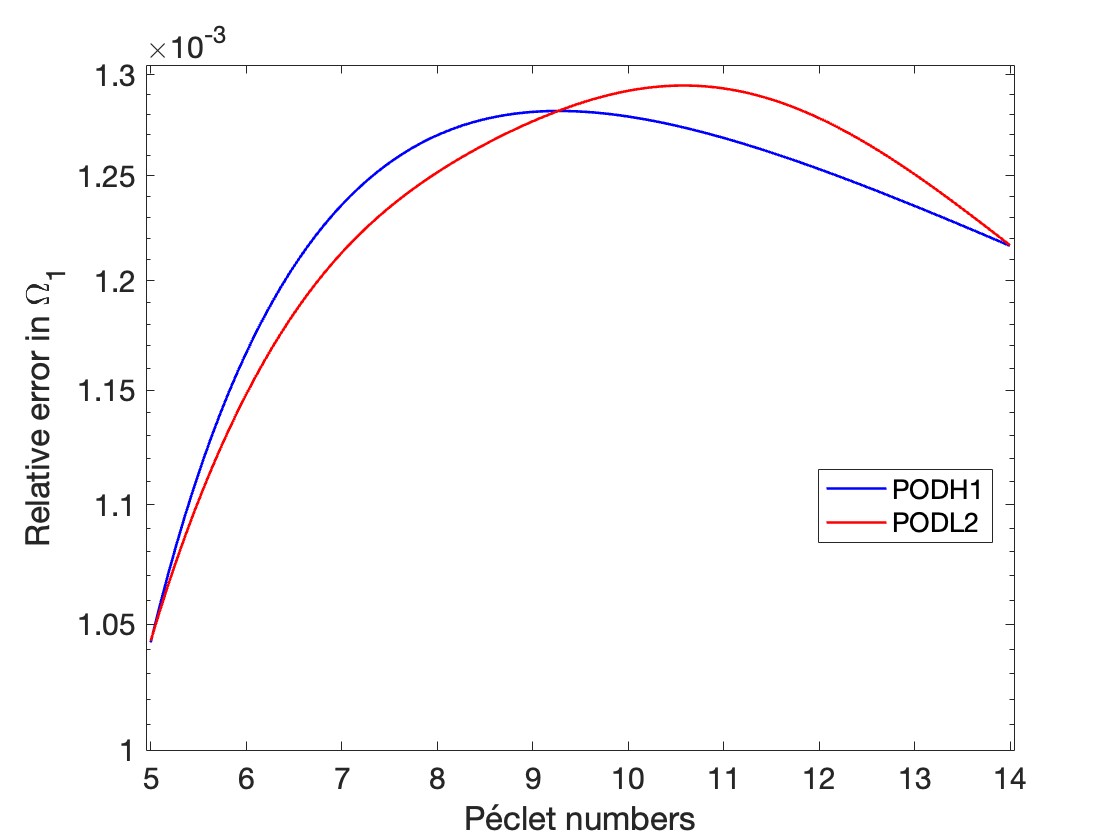}
        \includegraphics[width=8cm]{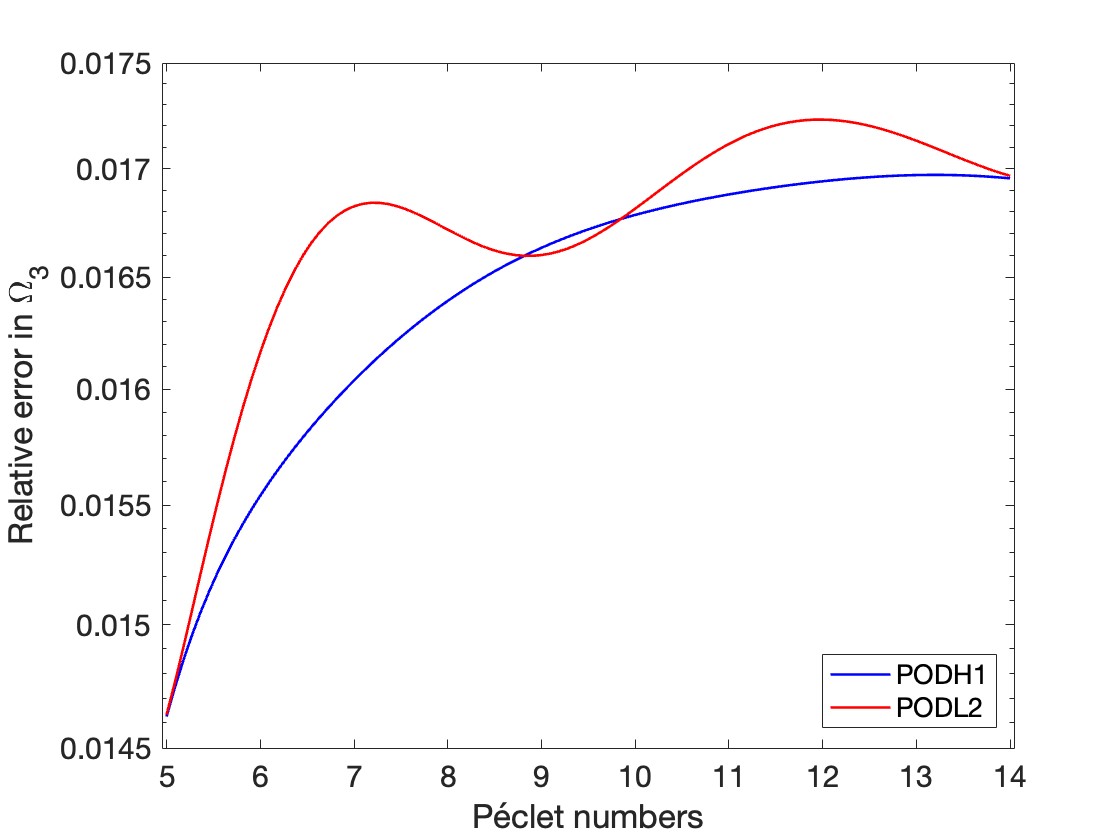}
    \caption{Reduced Schwarz algorithm. Relative errors in $L^2(\Omega_1)$ (left) and $L^2(\Omega_3)$ (right).}
    \label{graficaerror1_solap10}
\end{figure}


{\bf Extrapolation to parameter values outside training range.} We have finally tested the errors due to extrapolation of the reduced Schwarz procedure to values of $Pe$ outside the training set ${\mathcal D}$, we present the results in Figure \ref{graficaextrapol1}. We observe that the relative errors in $\Omega_1$ and $\Omega_3$ continue to be below $2\%$ for Péclet numbers approximately belonging to the interval $[3,17]$ and $[4,16]$, respectively. Beyond these values the errors experience a fast increase, as is standard in ROM of parametric problems.\\

{\bf Computational time speed-up} Regarding computational time, the off-line stage requires 1'43 CPU hours in a MacBook Pro M1 2020 laptop computer. However, the reduced Schwarz method provides a speed-up which is asymptotically increasing with respect to the number of nodes in the grid of $\Omega_2$. Indeed, the solution of problem \eqref{ecomega2} by iterative methods requires an amount of operations that scales as the number of nodes in the grid of $\Omega_2$, while the computation of the approximate trace mapping $\tilde \tau$ requires a number of operations that scales as the number of grid nodes in the vertical boundaries of $\Omega_2$.\\

\begin{figure}[h!]
    \centering
    \includegraphics[width=8cm]{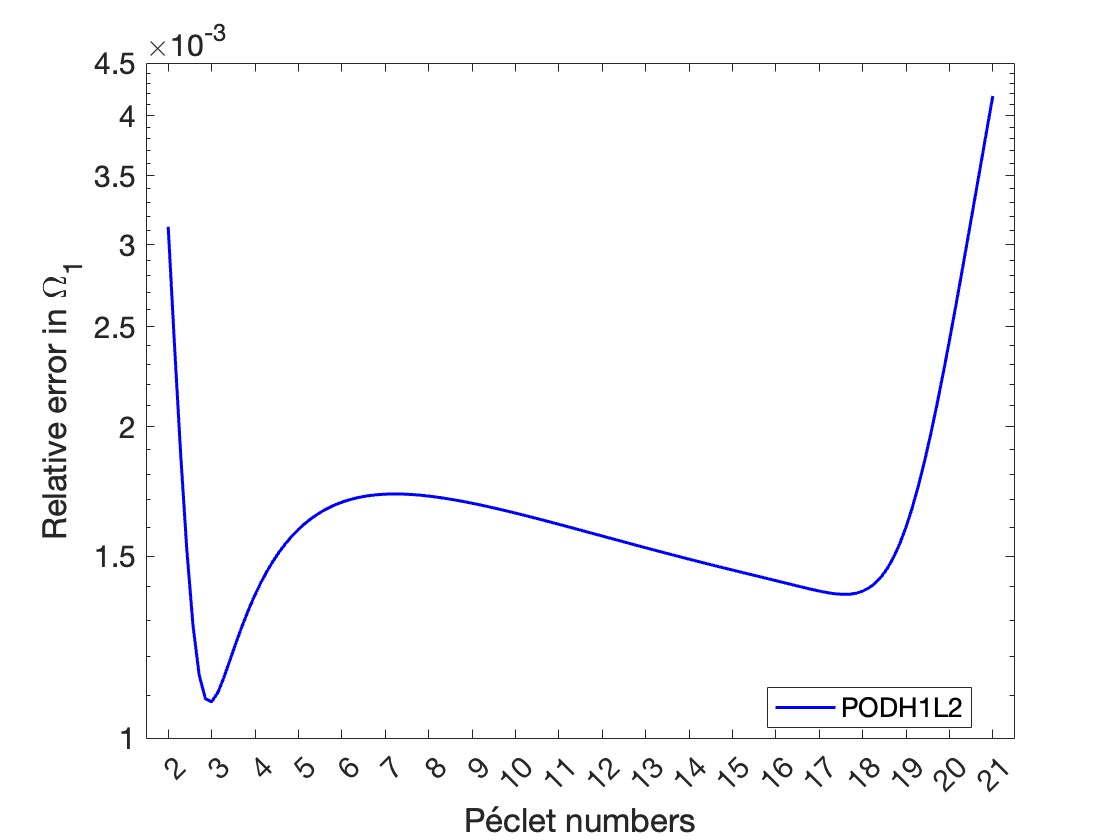}
        \includegraphics[width=8cm]{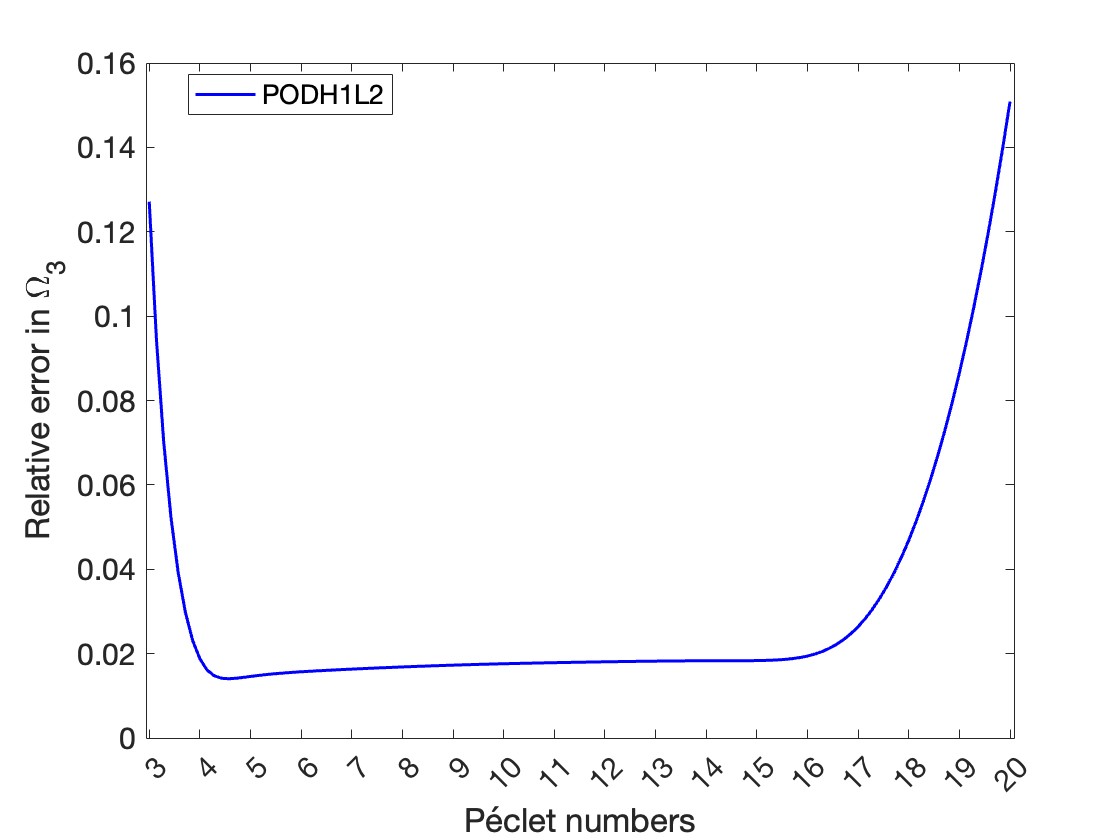}
    \caption{Extrapolation of reduced Schwarz algorithm. Relative errors in $L^2(\Omega_1)$ (left) and $L^2(\Omega_3)$ (right).}
    \label{graficaextrapol1}
\end{figure}

%

\section{Conclusions} \label{se:concl}
We present in this paper the results of a research motivated by the need of a very fast solution of thermal flow in solar receivers. These receivers are composed by a large number of parallel pipes with the same geometry, and connected flow at inflow and outflow boundaries. The temperature in each pipe satisfies an advection-diffusion equation with parametric-dependent boundary conditions. We are so led to the fast solution of parametric advection-diffusion equations on a pipe-like domain.\\

We have introduced a reduced Schwarz algorithm that skips the computation in a large part of the domain, requiring only the solution of the advection-diffusion equation in two small domains, close to the inflow and outflow boundaries. The computation of the temperature in the skep domain is replaced by reduced input temperature $\mapsto$ output temperature trace mapping. This reduced mapping is computed in an off-line stage. This stage uses the information of a number of full-order Schwarz algorithm runs to build reduced spaces to approximate the traces via POD analysis. An artificial neuronal network is used to compute the input $\mapsto$ output mapping of the latent variables.\\

We have performed an error analysis of the reduced Schwarz algorithm for quite general linear advection-reaction-diffusion equations. We have proved that the error is bounded in terms of the linearly decreasing error of the standard Schwarz algorithm, plus the error stemming from the reduction of the trace mapping. The last error is asymptotically dominant in the Schwarz iterative process.\\

We have obtained a reduction of the errors by three procedures. First, by enlarging the overlapping between subdomains. Second, by enriching the sampling in the off-line stage, that consists in solving the advection-diffusion problem in the domain to be omited, in order to enrich the training data on the discrete trace mapping. And third,  by POD analysis with respect to stronger norms than the $L^2$ norm on the active boundaries in the Schwarz iterative process, actually with respect to discrete $H^1$ norms. We obtain $L^2$ errors below $2\%$ with relatively small overlapping lengths ($12'5\%$), even in a somewhat larger parameter domain than the training one, acceptable for engineering applications.\\

At present is in progress the derivation of sampling techniques that would reduce the rather large off-line computational cost, in combination with the search for further procedures to decrease the error. Also, the application to non-linear thermal flow problems of the reduced Schwarz method, following the original motivation of this work.\\

{\bf Acknowledgements.} This research has been partially supported by the Marie Sklodowska-Curie Action "Accurate ROMs for Industrial Applications" 872442 - ARIA project, by the Spanish Government - Feder Fund Grant PID2021-123153OB-C21 and by the Spanish government CDTI-Misiones project CDTI project TRANSFER.

\end{document}